\pdfoutput=1
\RequirePackage{ifpdf}
\ifpdf % We~are running pdfTeX in pdf mode
\documentclass[pdftex]{sigma}
\else
\documentclass{sigma}
\fi

\numberwithin{equation}{section}

\newtheorem{Theorem}{Theorem}[section]
\newtheorem*{Theorem*}{Theorem}

\newtheorem{Lemma}[Theorem]{Lemma}

 { \theoremstyle{definition}

\newtheorem{Example}[Theorem]{Example}

\newtheorem*{question}{Main Question}
\newtheorem*{question2}{Technical Question}}

\usepackage{esint}
\usepackage{wasysym}
\usepackage{mathrsfs}
\def\CP{\mathbb{CP}}

\def\RR{\mathbb{R}}
\DeclareMathAlphabet{\zap}{OT1}{pzc}{m}{it}

\DeclareMathOperator{\Vol}{Vol}

\begin{document}
\allowdisplaybreaks

\renewcommand{\thefootnote}{}

\newcommand{\arXivNumber}{2302.12060}

\renewcommand{\PaperNumber}{027}

\FirstPageHeading

\ShortArticleName{Yamabe Invariants, Homogeneous Spaces, and Rational Complex Surfaces}

\ArticleName{Yamabe Invariants, Homogeneous Spaces, \\ and Rational Complex Surfaces\footnote{This paper is a~contribution to the Special Issue on Differential Geometry Inspired by Mathematical Physics in honor of Jean-Pierre Bourguignon for his 75th birthday. The~full collection is available at \href{https://www.emis.de/journals/SIGMA/Bourguignon.html}{https://www.emis.de/journals/SIGMA/Bourguignon.html}}}

\Author{Claude LEBRUN}

\AuthorNameForHeading{C.~LeBrun}

\Address{Department of Mathematics, Stony Brook University, Stony Brook, NY 11794-3651, USA}
\Email{\href{mailto:email@address}{claude@math.stonybrook.edu}}
\URLaddress{\url{http://www.math.stonybrook.edu/~claude/}}

\ArticleDates{Received February 23, 2023, in final form May 02, 2023; Published online May 07, 2023}

\Abstract{The Yamabe invariant is a diffeomorphism invariant of smooth compact manifolds that arises from the normalized Einstein--Hilbert functional. This article highlights the manner in which one compelling open problem regarding the Yamabe invariant appears to be closely tied to static potentials and the first eigenvalue of the Laplacian.}


\Classification{53C21; 53C18; 14J26; 58J50}

\renewcommand{\thefootnote}{\arabic{footnote}}
\setcounter{footnote}{0}

\section{Introduction}

The {\em Yamabe invariant} of a smooth compact $n$-manifold $M$, $n\geq 3$, is a
fascinating real-valued diffeomorphism invariant $\mathscr{Y}(M)$ that is defined by applying a minimax procedure to the
{\em normalized Einstein--Hilbert functional}
\begin{equation*}%\label{nehf}
\mathscr{E} \big(M^n, g \big) = \frac{\int_M s_g \,{\rm d}\mu_g}{(\int_M {\rm d}\mu_g)^{1-\frac{2}{n}}},
\end{equation*}
on the space of Riemannian metrics $g$; here $s_g$ and ${\rm d}\mu_g$ respectively denote the scalar curvature and
$n$-dimensional volume measure of any Riemannian metric $g$. One key reason
for studying this functional is that the critical points of $\mathscr{E}$ are exactly the {\em Einstein metrics}~\cite{bes},
or in other words the metrics of constant Ricci curvature. However, $\mathscr{E}$ is neither bounded above nor below,
so one can never find a critical point by minimizing or maximizing the functional. Nonetheless, Yamabe discovered that
$\mathscr{E}$ {does} become bounded below whenever one restricts it to a conformal class
\begin{equation*}
\gamma = [g_0]= \big\{ g = {\zap f} g_0 \mid {\zap f}\colon M\stackrel{C^\infty}{\longrightarrow} \RR^+\big\},
\end{equation*}
and his influential posthumous paper \cite{yamabe} moreover claimed to prove
 that the {\em Yamabe constant} of each conformal class $\gamma$, as defined by
 \begin{equation*}
Y(M, \gamma ) := \inf \mathscr{E} |_\gamma ,
\end{equation*}
 is always achieved by some minimizing metric.
 Although Yamabe's article made a perilous analytic mistake, a series of fundamental papers by
Trudinger \cite{trud}, Aubin \cite{aubyam}, and Schoen \cite{rick} eventually proved that such
minimizers, now called {\em Yamabe metrics}, do exist in every $\gamma$.

The {\em Yamabe invariant} (originally introduced by Kobayashi \cite{okob} and Schoen \cite{sch}, but under different names)
is defined by
\begin{equation*}
\mathscr{Y}(M) = \sup_\gamma Y(M, \gamma ) = \sup_\gamma \inf_{g\in \gamma} \mathscr{E} (M, g),
\end{equation*}
where one takes the infimum of the Einstein--Hilbert functional in each conformal class, and then takes the supremum of these infima over
 all conformal classes on $M$.
The existence of minimizers in every conformal class then allows us to reword this, more geometrically, as
\begin{equation*}
\mathscr{Y}(M) = \sup \{ s_g \mid g \text{ is a unit-volume Yamabe metric on } M \}.
\end{equation*}
This supremum is moreover finite, because one of Aubin's fundamental contributions to the Yamabe problem
is that any conformal class automatically satisfies
\begin{equation*}
{Y}\big(M^n, \gamma \big) \leq \mathscr{E}\big(S^n, g_{\text{\tiny unit}} \big)= n(n-1)\pi \left[
\frac{2\sqrt{\pi}}{\Gamma (\frac{n+1}{2})}
\right]^{2/n},
\end{equation*}
where $g_{\text{\tiny unit}}$ is the standard metric on the unit $n$-sphere; it therefore follows that any
smooth compact $n$-manifold satisfies
\begin{equation*}
\mathscr{Y}\big(M^n \big) \leq \mathscr{Y}\big(S^n \big)= \mathscr{E}\big(S^n, g_{\text{\tiny unit}} \big).
\end{equation*}
This ``mountain-pass'' definition thus attaches
a real number $\mathscr{Y}(M)$ to every smooth compact manifold.
Because $\mathscr{Y}(M)>0$ iff $M$ carries a metric of positive scalar curvature,
 the quest to compute Yamabe invariants represents a % natural,
 quantitative refinement of the problem
of determining which manifolds admit positive-scalar-curvature metrics.

While any Yamabe metric has constant scalar curvature, it must be emphasized that the converse is not generally true.
Nonetheless, there are two important basic results that run in the converse direction.
First, metrics of constant scalar curvature $s\leq 0$ are always Yamabe, and indeed are, up to constant rescalings,
 the unique Yamabe metrics in their conformal classes. Second, a theorem of Obata~\cite{obata} guarantees that, on a compact manifold~$M$, any Einstein metric~$g$
is necessarily a Yamabe metric; indeed, up to constant rescalings, it is the {\em only} constant-scalar-curvature metric in its conformal class,
except in the special case where $\big(M^n, [g]\big)$ is the standard sphere $\big(S^n, [g_{\text{\tiny unit}}]\big)$.

The fact that a metric of constant positive scalar curvature is not necessarily Yamabe makes it surprisingly difficult
to calculate the Yamabe invariant $\mathscr{Y}(M)$ in the positive case. As a~consequence, there are very few manifolds
where the Yamabe invariant is positive, known, and different from that of the sphere. By contrast, there are
large classes of manifolds with $\mathscr{Y}\leq 0$ where the precise value of the Yamabe invariant is currently known.
For example, the Yamabe invariant is known~\cite{albaleb,lky} for the underlying $4$-manifold~$M$ of any compact complex surface $\big(M^4,J\big)$
of Kodaira dimension $\neq -\infty$. One key ingredient is that Seiberg--Witten theory allows one to show~\cite{lno} that
if $\big(M^4,J\big)$ admits a K\"ahler--Einstein metric~$g$ of non-positive Ricci curvature, then
\begin{equation*}
\mathscr{Y} (M) = \mathscr{E}(M, g)=
-4\pi\sqrt{2 c_1^2 (M,J)}.
\end{equation*}
On the other hand, $\mathscr{Y}$ turns out to be invariant under blowing up or down in this context, so that \cite{albaleb,lno,lky}
\begin{equation*}
\mathscr{Y} \big(M\# k \overline{\CP}_2\big) =\mathscr{Y} (M)
\end{equation*}
for any complex surface $\big(M^4,J\big)$ of Kodaira dimension $\neq -\infty$.

A faint echo of the
 above story about K\"ahler--Einstein metrics can still be heard in the positive case. Indeed, two different arguments involving
 spin$^c$ Dirac operators~\cite{gl1,lcp2} show that the Fubini--Study metric $g_{\text{\tiny FS}}$ on $\CP_2$ similarly
achieves the manifold's Yamabe invariant:
\begin{equation*}
\mathscr{Y} (\CP_2) = \mathscr{E}(\CP_2, g_{\text{\tiny FS}})= 12\pi \sqrt{2} < 8 \pi \sqrt{6} = \mathscr{Y} \big(S^4\big).
\end{equation*}
By contrast, however, the other conformally K\"ahler, positive
Einstein metrics on $4$-manifolds \cite{chenlebweb,lebuniq,sunspot,page,ty} all have energies $\mathscr{E}$ less than corresponding the Yamabe invariants.
 Indeed, most of these manifolds are blow-ups of the complex projective plane, and in all these cases induction on an inequality due to Kobayashi \cite[Theorem 2]{okob} implies that
 \begin{equation*}
\mathscr{Y} \big(\CP_2\# k \overline{\CP_2}\big) \geq \min \big\{ \mathscr{Y} (\CP_2), \mathscr{Y} \big( \overline{\CP_2}\big), \ldots , \mathscr{Y} \big( \overline{\CP_2}\big)\big\} = \mathscr{Y} (\CP_2),
\end{equation*}
which is larger than the energy of any of the Einstein metrics in question.
In fact, there is only one smooth compact $4$-manifold that carries a conformally K\"ahler, positive Einstein metric that is not covered by this argument, namely the
spin manifold $\CP_1 \times \CP_1$. This case is harder, and demands a very different set of techniques. However,
 B\"ohm, Wang, and Ziller~\cite{bwz} were able to settle this outstanding case by showing that the K\"ahler--Einstein metric on $\CP_1 \times \CP_1$ can be perturbed to yield
nearby Yamabe metrics of slightly higher energy.

When $M$ is the underlying smooth $4$-manifold of a compact complex surface $\big(M^4, J\big)$ of Kodaira dimension $\neq -\infty$,
we have already pointed out that
the Yamabe invariant $\mathscr{Y}(M)$ is
unchanged if $M$ is blown up or blown down. This makes it natural to ask whether this pattern
also holds for the compact complex surfaces that can be obtained from~$\CP_2$ by some
 sequence of blow-ups and blow-downs; these are the
{\em rational complex surfaces} of the title. As a step in this direction, this article will focus
 on the following looser question:

\begin{question} If a compact complex surface $\big(M^4, J\big)$ is obtained from $\CP_2$ by blowing up and down, is it necessarily true that $\mathscr{Y}(M)\geq \mathscr{Y}(\CP_2)$?
\end{question}

Once again, the only hard case is $S^2 \times S^2$, so
a positive answer would immediately follow from an affirmative answer to the following:

\begin{question2} Let $h_2$ be the homogeneous metric on $S^2 \times S^2$ that is the Riemannian product of two standard round metrics on $S^2$, one of radius $1$,
and one of radius $1/\sqrt{2}$. Is~$h_2$ a~Yamabe metric?
\end{question2}

This metric, which satisfies $\mathscr{E}(S^2\times S^2, h_2) = \mathscr{Y} (\CP_2)$, also carries a static potential \cite{ambrozio,corvino,lafontaine},
and thus enjoys an interestingly different special status for the behavior of the scalar curvature.
 Our approach to the above technical question combines elements of the B\"ohm--Wang--Ziller argument with some simple observations about the spectrum of the
 Laplacian. I hope that Jean-Pierre will enjoy the way this connects these questions to his own investigations of the scalar curvature.

\section{The bass-note of a Yamabe metric}

One of Jean-Pierre Bourguignon's early contributions \cite{bes,bourguignon} to the study of the scalar curvature
was the discovery that, on a compact manifold $M^n$, the linearized behavior of the scalar curvature~$s$ at a constant-scalar-curvature
metric $g$ depends on whether $s/(n-1)$
is an eigenvalue of the Laplace--Beltrami operator $\Delta_g$. However, it is less often noted
 that this question also reveals important information about whether the given metric $g$ is a Yamabe metric:

\begin{Lemma}\label{symptom}
If $g$ is a Yamabe metric on a smooth compact manifold $M$ of dimension $n\geq 3$, then
\begin{equation}\label{eigen}
\lambda_1 \geq \frac{s}{n-1},
\end{equation}
where $\lambda_1$ is the smallest positive eigenvalue of the Laplace--Beltrami operator $\Delta_{g} = {\rm d}^*{\rm d}$,
and where $s$ is the scalar curvature of $g$.
\end{Lemma}
\begin{proof}
If $(M,g)$ is a compact Riemannian manifold of scalar curvature $s=\text{const}$, it suffices to prove the contrapositive statement that
\begin{equation*}
\lambda_1 < \frac{s}{n-1} \ \Longrightarrow \ \mbox{$g$ is not a Yamabe metric.}
\end{equation*}
To see this, first recall that, after setting $p:=\frac{2n}{n-2}$, so that
any metric $\widetilde{g}$ conformal to $g$ and its corresponding metric volume measure can be simultaneously expressed as
\begin{equation*}
\widetilde{g}= u^{p-2}g \qquad \mbox{and} \qquad \widetilde{{\rm d}\mu} = u^p \,{\rm d}\mu,
\end{equation*}
the scalar curvature $\widetilde{s} = s_{\tilde{g}}$ of the conformally rescaled metric can then be read off from the
{\em Yamabe equation}
\begin{equation}
\label{yammer}
\widetilde{s} u^{p-1} = \left[(p+2) \Delta + s\right] u.
\end{equation}
Now
assume that $s > (n-1) \lambda_1 > 0$, and suppose that $f\not\equiv 0$ belongs to the $\lambda_1$-eigenspace of the Laplacian. In particular, since
\begin{equation*}
\Delta f = \lambda_1 f,
\end{equation*}
where $\lambda_1 > 0$, it then follows that $\int_M f\,{\rm d}\mu_g=\frac{1}{\lambda_1}\int_M (\Delta f)\,{\rm d}\mu_g=0$.
Now set $u_t = 1 + t f$, and, for every small real number $t$, let $g_t = u_t^{p-2}g$.
Setting $V= \Vol (M,g)$ and letting $\fint = {V}^{-1} \int_M$, we then have
\begin{eqnarray*}
\mathcal{E}(g_t)&=&\frac{\int_M u_t \left[ (p+2)\Delta u_t + s u_t\right] {\rm d}\mu_g}{\left(\int_M u_{t}^p{\rm d}\mu_g\right)^{2/p}} \\
&=&\frac{V\fint \big[ t^2 (p+2)\lambda_1f^2 + s+ s t^2 f^2 \big] {\rm d}\mu_g}{\left(V\fint [1+t^2 \frac{p(p-1)}{2} f^2 +O(t^3)] {\rm d}\mu_g\right)^{2/p}} \\
&=&sV^{1-\frac{2}{p}} \frac{\left[ 1+ t^2 \big(1+ (p+2)\frac{\lambda_1}{s} \big)\fint f^2 {\rm d}\mu_g\right] }{\left[ 1+ \frac{p(p-1)}{2} t^2\fint f^2 {\rm d}\mu_g +O\big(t^3\big) \right]^{2/p}} \\
&=&sV^{2/n} \left[ 1+ t^2 \left(1+ (p+2)\frac{\lambda_1}{s} \right)\fint f^2 {\rm d}\mu_g\right] }{\left[ 1- (p-1) t^2\fint f^2 {\rm d}\mu_g +O\big(t^3\big) \right] \\
&=&sV^{2/n} \left[ 1+ t^2 \left((2-p)+ (p+2)\frac{\lambda_1}{s} \right)\fint f^2 {\rm d}\mu_g +O\big(t^3\big)\right] \\
&=&sV^{2/n} \left[ 1-\frac{4t^2}{n-2} \left(1- \frac{\lambda_1}{s/(n-1)} \right)\fint f^2 {\rm d}\mu_g +O\big(t^3\big)\right] \\
&<&sV^{2/n}= \mathcal{E}(g)
\end{eqnarray*}
for all sufficiently small $t \neq 0$. Thus $g$ is not a Yamabe metric, and the Lemma therefore follows by contraposition.
\end{proof}

\begin{Example}\label{zamp}
For any integer $k\geq 2$, let $(X^\ell , h)$ be a compact Einstein manifold of Ricci curvature $k-1$, and,
for an arbitrary real parameter
$\mathsf{t} \geq 1$, consider the Riemannian product metric
\begin{equation}
\label{squeeze}
h_{\mathsf{t}}= g_{\text{\tiny unit}} \oplus \mathsf{t}^{-1} h
\end{equation}
 on $S^k \times X^\ell$, where $g_{\text{\tiny unit}}$ is the standard metric on the unit $k$-sphere $S^k \subset \RR^{k+1}$.
 Because $\big(S^k \times X^\ell, h_{\mathsf{t}}\big)$ has Ricci curvature $\geq k-1$,
 Lichnerowicz's eigenvalue estimate \cite{lichnebook} tells us that it satisfies $\lambda_1 \geq k$; and
 this space therefore actually has $\lambda_1 =k$,
 since any linear functional $\RR^{k+1}\to \RR$,
 restricted to the unit $k$-sphere, and then pulled back to $S^k\times X^\ell$, is a Laplace eigenfunction
 with eigenvalue $k$. On the other hand, the scalar curvature
of $\big(S^k \times X^\ell, h_{\mathsf{t}}\big)$ is $s= (k-1) (k+\mathsf{t}\ell)$.
Thus, inequality~\eqref{eigen} reads $(k+\ell -1) k \geq (k-1) (k+\mathsf{t}\ell)$, and so is
 satisfied iff
$\mathsf{t}\in \big[1,\frac{k}{k-1}\big]$. Consequently, Lemma~\ref{symptom} implies that~$h_{\mathsf{t}}$ is {\em not}
 a Yamabe metric if $\mathsf{t}> \frac{k}{k-1}$. By contrast, when~$\mathsf{t}=1$,
the metric $h_1$ is Einstein, and so, by Obata's theorem~\cite{obata}, is certainly a~Yamabe metric.
\end{Example}

These last observations fit into an interesting broader discussion. On a compact $n$-mani\-fold~$M$, $n \geq 3$,
Obata's theorem guarantees that any {Einstein} metric $g$
is actually a Yamabe metric, so Lemma \ref{symptom} implies that \eqref{eigen} holds for any compact Einstein manifold. However,
 Lichnerowicz gave a Bochner--Weitzenb\"ock proof \cite{lichnebook} in this Einstein case that is not only more elementary, but also sharpens the result, as it shows that
 any compact Einstein manifold $\big(M^n,g\big)$ for which equality holds
 in~\eqref{eigen} must, up to isometry and rescaling, be the standard $n$-sphere $\big(S^n, g_{\text{\tiny unit}}\big)$.
 By contrast, however,
we will now see that equality holds in~\eqref{eigen} for many other Yamabe metrics.

\begin{Lemma} \label{plain}
On a given compact $n$-manifold $M$, let
 $g(t)= [u(t)]^{p-2}g$ be a smooth one-parameter family of metrics in some conformal class $[g]$. If all of these metrics have the same
constant scalar curvature $s= \text{const}$, and if $u(t) = 1 + vt + O\big(t^2\big)$, then the function $v= \dot{u}(0)$ satisfies
\begin{equation*}
\Delta v = \frac{s}{n-1} v
\end{equation*}
with respect to the background metric $g=g(0)$.
\end{Lemma}
\begin{proof} Under these circumstances, the
 Yamabe equation~\eqref{yammer} becomes
\begin{equation*}
s [u(t)]^{p-1} = (p+2) \Delta u(t) + su(t),
\end{equation*}
and differentiating with respect to the parameter $t$ thus yields
\begin{equation*}
s(p-1) [u(t)]^{p-2} \dot{u}(t) = (p+2)\Delta \dot{u}(t)+ s\dot{u}(t).
\end{equation*}
Setting $t=0$, where $u(0)=1$ and $\dot{u}(0)=v$ by assumption, we obtain
\begin{equation*}
\Delta v = s \left(\frac{p-2}{p+2}\right) v,
\end{equation*}
and the substitution $p :=\frac{2n}{n-2}$ therefore, upon simplification, yields
\begin{equation*}
 \Delta v = \frac{s}{n-1} v.
\end{equation*}
 If $v\not\equiv 0$, this function is therefore a Laplace--Beltrami
 eigenfunction, with eigenvalue~$\frac{s}{n-1}$.
\end{proof}

\begin{Lemma} \label{evident}
Let $g$ be a Yamabe metric on a smooth compact manifold $M$ of dimension $n\geq 3$, and suppose
that $(M,g)$ carries a {conformal} Killing field $\xi$ that is {not} actually Killing. Then $(M,g)$ satisfies
\begin{equation*}
\lambda_1 = \frac{s}{n-1},
\end{equation*}
and so saturates the bound provided by inequality \eqref{eigen}.
\end{Lemma}
\begin{proof} If $\Phi_t$ is the $1$-parameter family of diffeomorphisms generated by $\xi$, the pull-backs $ \Phi_t^*g= [u (t)]^{p-2}g$
 form a family of constant-scalar-curvature metrics as in Lemma~\ref{plain}. Moreover
 $v:= \dot{u}(0)\not\equiv 0$, because $(p-2) vg= \mathscr{L}_\xi g\not\equiv 0$.
Since $\Delta v = \frac{s}{n-1} v$ by Lemma \ref{plain}, it thus follows that
$\frac{s}{n-1}$ belongs to the Laplace spectrum of $(M,g)$. Next, observe that, since the $\Phi_t^*g$ provide us with non-homothetic,
 constant-scalar-curvature
metrics in the conformal class $[g]$, the scalar curvature $s$ of $g$ must be positive. But since
$\lambda_1$ is by definition the smallest positive Laplace eigenvalue of $(M,g)$, this implies that $\lambda_1 \leq \frac{s}{n-1}$.
However, since $g$ is by assumption a Yamabe metric, Lemma \ref{symptom} also tells us that
$\lambda_1 \geq \frac{s}{n-1}$.
It therefore follows that $\lambda_1 = \frac{s}{n-1}$, as claimed.
\end{proof}

\begin{Lemma}\label{test}
 Let ${g}$ be a unit-volume homogeneous metric on a smooth compact $n$-manifold $M$, $n\geq 3$, and let
 $\tilde{g}\in [{g}]$ be a Yamabe metric of unit volume in the conformal class of ${g}$. If $\tilde{g}\neq {g}$, then the scalar curvature $s$
 and the first Laplace eigenvalue $\lambda_1$ of $\tilde{g}$ must satisfy
 \begin{equation*}
 \lambda_1 = \frac{s}{n-1}.
\end{equation*}
 \end{Lemma}
\begin{proof} By hypothesis, there is a compact connected Lie group $\mathsf{G}$ that acts transitively and isometrically on $(M, {g})$.
The Lie algebra $\mathfrak{g}$ of $\mathsf{G}$ thus acts on $M$ by Killing fields of ${g}$, and these are therefore all conformal Killing fields of
the Yamabe metric $\tilde{g}$. However, since $\tilde{g}\neq {g}$ by assumption, and since the homogeneous metric ${g}$ is the unique unit-volume,
$\mathsf{G}$-invariant metric in its conformal class, there must be some $\xi \in \mathfrak{g}$ which is not
a Killing field for $\tilde{g}$. The claim thus follows by applying Lemma \ref{evident} to $\tilde{g}$. \end{proof}

\begin{Theorem} Let $M= \mathsf{G}/\mathsf{H}$ be a compact homogeneous manifold of dimension $n \geq 3$, where~$\mathsf{G}$ is a compact connected
Lie group and
$\mathsf{H}\subset \mathsf{G}$ is a compact Lie subgroup. Consider the finite-dimensional manifold
\begin{equation*}
\mathscr{M} = \left\{ \text{unit-volume $\mathsf{G}$-invariant metrics } g \text{ on } M \mid \lambda_1 > \frac{s}{n-1} \right\},
\end{equation*}
and define subsets $\mathscr{U}\subset \mathscr{V}\subset \mathscr{M}$ by
\begin{equation*}
\mathscr{V} = \{ g \in \mathscr{M} \mid g \text{ is a Yamabe metric} \},\quad \text{\em and}
\end{equation*}
\begin{equation*}
\mathscr{U} = \{ g \in \mathscr{M} \mid g \text{ is the {\sf unique} unit-volume Yamabe metric in } [g]\}.
\end{equation*}
Then $\mathscr{U}\subset\mathscr{M}$ is open, while $\mathscr{V}\subset \mathscr{M}$ is closed.
\end{Theorem}
\begin{proof} The complement $\mathscr{U}^\complement = \mathscr{M} -\mathscr{U}$ of $\mathscr{U}$ consists
of those $g\in \mathscr{M}$ for which there is a~unit-volume Yamabe metric $\tilde{g}\in [g]$ with $\tilde{g}\neq g$.
By Lemma~\ref{test}, this complement can therefore also be described as
\begin{equation*}
\mathscr{U}^\complement =\left\{ g\in \mathscr{M} \mid \exists \text{ unit-volume Yamabe $\tilde{g}\in [g]$ for which } \lambda_1 = \frac{s}{n-1}\right\},
\end{equation*}
since, by definition, any $g\in \mathscr{M}$ satisfies $\lambda_1 \neq \frac{s}{n-1}$.

We will now use this observation to prove that $\mathscr{U}^\complement$ is closed. Indeed, suppose that $g_j\in \mathscr{U}^\complement$
is a sequence of unit-volume homogeneous metrics with $g_j \to g$ for some $g\in \mathscr{M}$.
By definition, we may then choose a sequence $\tilde{g}_j\in [g_j]$ of unit-volume, non-homogeneous
 Yamabe metrics in the corresponding conformal classes.
However, since $g\in\mathscr{M}$ satisfies $\lambda_1 > \frac{s}{n-1}$ by hypothesis, it follows that
$(M,g)$ is certainly not a constant-sectional-curvature $n$-sphere, and a~compactness argument~\cite[p.~716]{bwz}
therefore guarantees that a subsequence
$\tilde{g}_{j_i}$ of these Yamabe metrics must converge in $C^{2}$ to some unit-volume Yamabe metric $\tilde{g}\in [g]$; cf.\ \cite[Proposition~2.1]{sch}.
It then follows that $\lambda_1 (\tilde{g}_{j_i}) \to \lambda_1 (\tilde{g})$ and $s(\tilde{g}_{j_i})\to s(\tilde{g})$.
However, since $\lambda_1 (\tilde{g}_{j_i}) = \frac{s}{n-1}(\tilde{g}_{j_i})$, this then implies that $\lambda_1 (\tilde{g})= \frac{s}{n-1} (\tilde{g})$,
 too. But $\lambda_1 ({g})> \frac{s}{n-1} (g)$ by our definition of $\mathscr{M}$, so it follows that~$\tilde{g}\neq g$. Hence
 $g\in \mathscr{U}^\complement$. We have
thus shown that $\mathscr{U}^\complement$ is closed, and hence that $\mathscr{U}$ is open.

On the other hand \cite[Proposition~4.31]{bes}, the Yamabe constant $Y(M, [g])$ is a continuous function of $g$, as is, more obviously, the scalar curvature $s_g$. Since we have
normalized each $g\in \mathscr{M}$ to have unit volume, the subset $\mathscr{V}\subset \mathscr{M}$ is exactly
the zero locus of the continuous function $Y(M, [g])-s_g$. It thus follows that $\mathscr{V}$ is closed, as claimed.
\end{proof}

Dropping the convenient but unnecessary unit-volume constraint, and recalling our discussion of \eqref{squeeze} in Example \ref{zamp}, we now apply this theorem to
the family $\big(S^2 \times S^2, h_{\mathsf{t}}\big)$, $\mathsf{t}\in [1,2]$, where $h_{\mathsf{t}}
=g_{\text{\tiny unit}}\oplus \mathsf{t}^{-1} g_{\text{\tiny unit}}$.
It follows that there is a maximal $\mathsf{T} \in (1,2]$
such that
$h_{\mathsf{t}}$ is, up to homothety,
the {unique} Yamabe metric in $[h_{\mathsf{t}}]$ for every $\mathsf{t}\in [1,\mathsf{T})$. However, it then also follows, by continuity, that
$h_{\mathsf{T}}$ is itself a Yamabe metric. To get an affirmative answer to the
the questions posed in the Introduction, it would thus suffice to show, if $\mathsf{T} < 2$, that $h_{\mathsf{T}}$ is necessarily, up to homothety,
the {\sf unique} Yamabe metric in $[h_{\mathsf{T}}]$. The fact that $h_{\mathsf{T}}$ is a Yamabe, K\"ahler metric on a $4$-manifold does seem to offer some hope of proving just this,
especially in light of the suggestive results of~\cite{gursky, gl1}. Nonetheless, this problem remains, for the present, just a topic for further investigation.

\section{Static potentials}

We now conclude this article with a few remarks relating our discussion to the fascinating topic of static potentials, which originally arose
from the work of Bourguignon \cite{bourguignon} and Fischer--Marsden~\cite{fischer-marsden}, and later led to the beautiful results
of Lafontaine~\cite{lafontaine} and others~\cite{ambrozio,corvino,seshadri}.
The motivating question here is to ask for the precise circumstances under which the linearization
\begin{equation}\label{lily}
Ds\colon \ \Gamma \big({\odot}^2 T^*M\big)\to C^{\infty}(M)
\end{equation}
 of the scalar curvature, at a given Riemannian metric $g$ on a smooth compact $n$-manifold $M$,
 fails to be surjective. Since the $t$-derivative of the scalar curvature $s$, at $t=0$, for a family
 $g+ t h$ is
 \begin{equation*}
\dot{s}=\Delta \big(h_{ab}g^{ab}\big) + \nabla^a\nabla^bh_{ab} - r^{ab}h_{ab},
\end{equation*}
integration by parts reveals that
the co-kernel of~\eqref{lily} exactly consists of functions $f$ that satisfy the overdetermined linear equation
 \begin{equation}\label{static}
0= (\Delta f)g_{ab} + \nabla_a \nabla_b f - fr_{ab}.
\end{equation}
A solution $f\not\equiv 0$ of \eqref{static} is called a {\em static potential}
on $(M,g)$, in order to highlight the remarkable fact \cite[Corollary~9.107, with $p=1$, $\hat{\lambda}=0$]{bes} that
\eqref{static} is precisely equivalent to demanding that
\begin{equation}\label{warp}
\widehat{g}=g+f^2 \,{\rm d}t^2
\end{equation}
 be an Einstein metric on $(M-Z) \times S^1$, where $Z:=f^{-1}(0)$.

One simple strategy for guessing a solution $f$ of \eqref{static} would be to take the function $f\not\equiv 0$ to be constant.
However, inspection reveals that this ansatz works
 if and only if $(M,g)$ is Ricci-flat. While this does show that the linearized scalar curvature \eqref{lily} fails to be surjective on any compact Ricci-flat manifold,
 there is really nothing else to say about this ``trivial'' case. {We will therefore assume henceforth that $f$ is non-constant}.

Next, notice that taking the trace of \eqref{static} yields the identity
\begin{equation}
\label{basso}
\Delta f = \frac{s}{n-1} f .
\end{equation}
We may therefore rewrite \eqref{static}
as
\begin{equation}
\label{thingone}
\nabla_a \nabla_b f = \varrho_{ab} f,
\end{equation}
where $\varrho$ is defined, in terms of the scalar curvature $s$ and the Ricci tensor $r$ or trace-free Ricci tensor $\mathring{r}$, to be
\begin{equation*}
\varrho_{ab}:= r_{ab} - \frac{s}{n-1} g_{ab}= \mathring{r}_{ab} - \frac{s}{n(n-1)} g_{ab}.
\end{equation*}
Equivalently, $f$ satisfies the linear second-order ODE
\begin{equation}\label{geode}
\frac{{\rm d}^2}{{\rm d}t^2}f = \varrho (\mathsf{v} , \mathsf{v})f
\end{equation}
along every unit-speed geodesic, where $\mathsf{v}={\rm d}/{\rm d}t$ denotes the unit tangent vector.
Since we have assumed that our connected Riemannian manifold $(M,g)$ is compact, and hence
geodesically complete, and since
$f\not\equiv 0$, it follows that $df$ must be non-zero at every point of $Z=f^{-1}(0)\subset M$;
otherwise, any $q\in M$ could be joined by some geodesic segment to a point $p\in Z$
at which the initial-value pair $(f(p), f^\prime (p))$ for \eqref{geode} vanished, thereby forcing~$f$ to vanish at~$q$, too.
In particular, $Z$ is a smooth hypersurface. Moreover, $Z$ is totally geodesic, because the same initial-value argument for~\eqref{geode} shows that any geodesic that is initially tangent to~$Z$
 must remain confined to $Z=f^{-1}(0)$.
Finally, notice that, since
 \begin{equation*}
\nabla_a |\nabla f|^2 = {2}\big(\nabla^b f\big) \nabla_a \nabla_b f= 2f\varrho_{ab}\big(\nabla^b f\big)
\end{equation*}
vanishes along $Z$, the length of the normal vector field~$\nabla f$ is actually constant on each connected component of~$Z$.

We next apply the divergence operator $-\nabla^a$ to both sides of \eqref{static}. Remembering Bochner's formula for the Hodge Laplacian on $1$-forms, and recalling that $d$ commutes with the Hodge Laplacian, the Bianchi identity then yields
\begin{align*}
0&= -\nabla^a (g_{ab}\Delta f) - \nabla^a\nabla_a \nabla_b f + \nabla^a(fr_{ab})
 -\nabla_b \Delta f + \nabla^*\nabla (\nabla_b f )+ r_{ab}\nabla^af + f\nabla^ar_{ab}
\\&=-\nabla_b ({\rm d}+{\rm d}^*)^2 f + ({\rm d}+{\rm d}^*)^2 \nabla_b f+ \tfrac{1}{2} f\nabla_b s
\\&= \tfrac{1}{2} f\nabla_b s .
\end{align*}
This shows that $\nabla s=0$ on the set $M-Z$ where $f \neq 0$. However, since $Z$ is just a hypersurface,
$M-Z\subset M$ is dense, and it therefore follows by continuity that $\nabla s=0$ everywhere; and since~$M$ has been assumed to be connected, this then implies that the scalar curvature $s$ is constant.
Consequently, \eqref{basso} now tells us that $f$ is an eigenfunction of the Laplacian, with eigenvalue~$\frac{s}{n-1}$.
This in particular gives the scalar curvature a Rayleigh-quotient interpretation:
\begin{equation*}
\frac{s}{n-1} = \frac{\int_M f\Delta f\, {\rm d}\mu}{ \int_M f^2\, {\rm d}\mu} = \frac{\int_M|\nabla f|^2\, {\rm d}\mu}{ \int_M f^2\, {\rm d}\mu}.
\end{equation*}
Our assumption that $f$ is {\em non-constant} therefore constrains the scalar curvature $s$ to be {positive},
and so forces the average value of $f$ to vanish:
\begin{equation*}
\fint_M f\, {\rm d}\mu = \frac{n-1}{s}\fint_M \Delta f\, {\rm d}\mu = \frac{n-1}{sV}\int_M \Delta f\, {\rm d}\mu =0.
\end{equation*}
Since the smooth function $f$ must achieve its average value, there must be some point $p\in M$ where $f$ vanishes, and we
therefore conclude that
 the totally geodesic hypersurface $Z= f^{-1}(0)$ is necessarily non-empty.

A particularly attractive special case of the above occurs when $Z$ is connected.
The Mayer--Vietoris sequence then guarantees that $M-Z$ must have exactly two connected components, namely the two
 regions respectively defined by $f>0$ and $f< 0$. On the other hand, the connectedness of $Z$ implies that $|\nabla f|$ is a non-zero constant on $Z$,
so the linearity of \eqref{static} allows us, by replacing~$f$ with a constant multiple, to arrange that
$|\nabla f|\equiv 1$ along $Z$. Setting $M_+ = f^{-1} ([0,\infty))$, we can now construct a smooth compact connected $(n+1)$-manifold
by setting $N= \big( M_+\times S^1\big)/{\sim}$,
where the equivalence relation
collapses $\partial \big( M_+\times S^1\big)= Z\times S^1$ down to $Z$, and~\eqref{warp} then defines a smooth Einstein metric $\widehat{g}$ on $N^{n+1}$. Of course,
replacing~$f$ with~$-f$ also gives us a second Einstein $(n+1)$-manifold, which in principle might not be isometric to the first. Note that the connectedness of~$Z$ plays an important role in this construction, since otherwise $|\nabla f|$ could take different values on different connected components, leading to unavoidable
edge-cone singularities in $\big(N,\widehat{g}\big)$; for details, see~\cite{ambrozio}.

Following in the tracks of Lafontaine \cite[Section~B]{lafontaine}, we next observe
that the metrics of~\eqref{squeeze}, for the critical parameter value $\mathsf{t}= \frac{k}{k-1}$,
provide us with examples of manifolds that carry static potentials.
Thus, for any positive integer $k$, let $\big(S^k, g_{\text{\tiny unit}}\big)$ be the standard
unit $k$-sphere, and let $f\colon S^k \to \RR$ be the restriction of the Euclidean $x^1$ coordinate to $S^k \subset \RR^{k+1}$.
The gradient $\eta= \nabla f$ is then a conformal Killing field on~$S^k$ with
\begin{equation*}
\mathscr{L}_\eta g_{\text{\tiny unit}} = -2f g_{\text{\tiny unit}},
\end{equation*}
and one therefore has
\begin{equation*}
\nabla \nabla f = - f g_{\text{\tiny unit}}
\end{equation*}
on $\big(S^k, g_{\text{\tiny unit}}\big)$. Now let $\big(X^{\ell}, \widehat{h}\big)$ be a compact $\ell$-dimensional Einstein manifold of Ricci curvature~$k$,
and set
\begin{equation*}
\big(M^{n}, g\big) = \big(S^k, g_{\text{\tiny unit}}\big) \times \big(X^\ell, \widehat{h}\big),
\end{equation*}
so that $n = k+\ell$ and $g=g_{\text{\tiny unit}} \oplus \widehat{h}$.
We now reinterpret $f$ as a function $f\colon M \to \RR$, by pulling it back from the first factor
of the Riemannian product. In this context, the gradient $\eta= \nabla f$ of this function
then satisfies
\begin{equation*}
\mathscr{L}_\eta g = -2f g_{\text{\tiny unit}},
\end{equation*}
so that we still have
\begin{equation}
\label{cueball}
\nabla \nabla f = - f g_{\text{\tiny unit}}
\end{equation}
on $(M, g)$. On the other hand, $\big(S^k \times X,g\big)$ has Ricci tensor
\begin{equation*}
r = (k-1) g_{\text{\tiny unit}} \oplus k \widehat{h}=kg-g_{\text{\tiny unit}}
\end{equation*}
and hence has scalar curvature
\begin{equation*}
s = k(k-1) + k\ell = k (n-1).
\end{equation*}
As a consequence, $M=S^k \times X$ has
\begin{equation*}
\varrho := r - \frac{s}{n-1}g= (k g - g_{\text{\tiny unit}}) - k g = -g_{\text{\tiny unit}}.
\end{equation*}
Equation \eqref{cueball} can therefore be written as
\begin{equation*}
\nabla \nabla f = \varrho f,
\end{equation*}
so that $f$ satisfies \eqref{thingone}, and is therefore a static potential on $(M,g)$.

Finally, what are the Einstein $(n+1)$-manifolds that arise from \eqref{warp} in the context of these specific examples? They are in fact just the
the Riemannian products
\begin{equation*}
\big(N, \widehat{g}\big)= \big(S^{k+1} , g_{\text{\tiny unit}}\big) \times \big(X, \widehat{h}\big).
\end{equation*}
For example, to get the metric $h_2$ on $S^2 \times S^2$ featured in the technical question of the introduction,
one first divides a product Einstein metric on $S^3 \times S^2$ by a rotation
of the $S^3$ factor around an equatorial~$S^1$, and then doubles the result. In any such example, the upstairs metric~$\widehat{g}$ is Yamabe by Obata's theorem,
and this then implies an interesting weighted Sobolev--Kondrakov inequality downstairs. Unfortunately, however, this inequality does not by itself appear sufficient to imply that the downstairs metric $g$ is itself Yamabe.

\subsection*{Acknowledgements}
It is a pleasure to dedicate this article to Jean-Pierre Bourguignon, on the occasion of
his seventy-fifth birthday. This research was supported in part by NSF grant DMS--2203572.

\pdfbookmark[1]{References}{ref}
\LastPageEnding

\end{document}